\documentclass{amsart}
\usepackage{amsmath, amssymb, amsthm}
\usepackage{mathtools}
\usepackage{amscd}
\usepackage{graphicx} 
\usepackage{enumerate}
\usepackage{todonotes}

\def\CC{{\mathbb C}}

\def\PP{{\mathbb P}}

\def\RR{{\mathbb R}}
\def\TT{{\mathbb T}}

\def\Qbar{\overline{\mathbb Q}}

\def\0{{\mathbf 0}}
\def\1{{\mathbf 1}}

\def\Gal{\mathrm{Gal}}

\def\min{\mathrm{min}}

\usepackage[all]{xy}

\newtheorem{thm}{Theorem}

\newtheorem{lemma}[thm]{Lemma}

\newtheorem*{thm*}{Theorem}
\newtheorem*{alg*}{Algorithm}
\newtheorem*{lemma*}{Lemma}

\theoremstyle{remark}

\newtheorem*{rmk*}{Remark}

\newtheorem*{notation*}{Notation}
\newtheorem{example}[thm]{Example}
\newtheorem*{example*}{Example}

\theoremstyle{definition}

\newtheorem*{defn*}{Definition}

\newcommand{\mybf}{\mathbb}

\newcommand{\bP}{\mybf{P}}
\newcommand{\bR}{\mybf{R}}

\newcommand{\bC}{\mybf{C}}

\newcommand{\bZ}{\mybf{Z}}
\newcommand{\bQ}{\mybf{Q}}

\newcommand{\al}{\alpha}

\providecommand{\abs}[1]{\lvert#1\rvert}

\newcommand{\ON}[1]{\operatorname{#1}}

\newcommand{\ra}{\rightarrow}

\newcommand{\logabs}[1]{\log\,\abs{#1}}

\def\talltareesidedbox#1{\setbox0=\hbox{$#1$}\dimen0=\wd0 \advance\dimen0 by3pt\rlap{\hbox{\vrule height10pt width.4pt
 depth2pt \kern-.4pt\vrule height10.4pt width\dimen0 depth-10pt\kern-.4pt \vrule height10pt width.4pt depth2pt}}
 \relax \hbox to\dimen0{\hss$#1$\hss}}
\def\tareesidedbox#1{\setbox0=\hbox{$#1$}\dimen0=\wd0 \advance\dimen0 by3pt\rlap{\hbox{\vrule height8pt width.4pt
 depth2pt \kern-.4pt\vrule height8.4pt width\dimen0 depth-8pt\kern-.4pt \vrule height8pt width.4pt depth2pt}}
\relax \hbox to\dimen0{\hss$#1$\hss}}
\def\shorttareesidedbox#1{\setbox0=\hbox{$#1$}\dimen0=\wd0 \advance\dimen0 by3pt\rlap{\hbox{\vrule height7pt width.4pt
 depth2pt \kern-.4pt\vrule height7.4pt width\dimen0 depth-7pt\kern-.4pt \vrule height7pt width.4pt depth2pt}}
 \relax \hbox to\dimen0{\hss$#1$\hss}}

\title[Height bounds]{Heights bounds for algebraic numbers satisfying splitting conditions}
 
\author[Fili]{Paul Fili}
\address{Department of Mathematics\\ Oklahoma State University, Stillwater, OK 74078}
\email{paul.fili@okstate.edu}
\author[Pritsker]{Igor Pritsker}
\thanks{Research of Igor Pritsker was partially supported by the National Security Agency (grant H98230-15-1-0229).}
\address{Department of Mathematics\\ Oklahoma State University, Stillwater, OK 74078}
\email{igor@math.okstate.edu}

\subjclass[2010]{11G50, 11R06, 37P30}
\keywords{Energy integrals, Weil height, totally real, totally $p$-adic.}
\date{\today}

\begin{document}

\begin{abstract}
 In an earlier work, the first author and Petsche solved an energy minimization problem for local fields and used the result to obtain lower bounds on the height of algebraic numbers all whose conjugates lie in various local fields, such as totally real and totally $p$-adic numbers. In this paper, we extend these techniques and solve the corresponding minimization programs for real intervals and $p$-adic discs, obtaining several new lower bounds for the height of algebraic numbers all of whose conjugates lie in such sets.
\end{abstract}

\maketitle

\allowdisplaybreaks[2]

\section{Introduction}

Let $h$ denote the absolute logarithmic Weil height on algebraic numbers. It is well-known that if an $\al\in\Qbar$ satisfies some sort of prime ideal splitting conditions, one can compute lower bounds on the height of the number. The first result of this kind dates back to Schinzel \cite{SchinzelTotReal}, who proved that if $\al\neq 0,\pm 1$ is totally real, that is, if the set of Galois conjugates of $\al$ lies entirely in $\bR$, then
\[
 h(\al) \geq \frac12 \log \bigg(\frac{1+\sqrt{5}}{2}\bigg).
\]
Bombieri and Zannier \cite{BombieriZannierNote} proved a similar result for totally $p$-adic numbers, specifically, if $L_S$ denotes the field of all numbers whose conjugates lie in the $p$-adic field $\bQ_p$ for all $p$ in a set $S$ of non-archimedean rational primes, then
\[
 \liminf_{\al\in L_S} h(\al) \geq \frac 12 \sum_{p\in S} \frac{\log p}{p+1}.
\]
(Bombieri and Zannier proved similar results for finite extensions of $\bQ_p$, but for simplicity we will first state all results for the moment in the totally $p$-adic setting.)

Using potential theoretic techniques on the Berkovich projective line, the first author and Petsche \cite{F-P-EIOLF} managed to improve on these results at the non-archimedean places and to allow simultaneous $p$-adic and totally real splitting conditions. Specifically, they proved \cite[Theorem 3]{F-P-EIOLF} that if $S$ is a nonempty subset of rational primes, and $L_S$ denotes the subfield of $\Qbar$ consisting of all those $\alpha\in\Qbar$ such that $\alpha$ is totally $p$-adic for all primes $p\in S$, and $\alpha$ is totally real if $\infty\in S$, then
\begin{equation}\label{SimpleGlobalThmBound}
\liminf_{\alpha\in L_S}h(\alpha)\geq
\begin{cases}
\displaystyle \frac{1}{2} \sum_{p\in S} \frac{p\log p}{p^2-1} & \text{if }\infty\not\in S\\
\displaystyle \frac{1}{2} \sum_{\substack{p\in S\\ p\nmid\infty}} \frac{p\log p}{p^2-1}  + \frac{7\zeta(3)}{4\pi^2} & \text{if }\infty\in S.
\end{cases}
\end{equation}

In this paper, we use similar potential theoretic techniques to extend the results of the first author and Petsche to the case where all of the conjugates of a number lie in a specific real or $p$-adic interval or ball, respectively. 

Before stating our results, let us first set some notation. For $L$ a local field with a given absolute value $\abs{\cdot}$ and $\mu$ a Borel measure on $\bP^1(L)$, we let
\begin{equation}\label{eqn:energy-integral}
I_\delta(\nu)=\iint_{\PP^1(L)\times\PP^1(L)}-\log\delta(x,y)\,d\nu(x)\,d\nu(y),
\end{equation}
where $\delta:\PP^1(L)\times\PP^1(L)\to\RR$ is the spherical metric defined by 
\begin{equation*}
\delta(x,y) =\frac{|x_0y_1-y_0x_1|}{\max\{|x_0|,|x_1|\}\max\{|y_0|,|y_1|\}}
\end{equation*}
for $x=(x_0:x_1)$ and $ y=(y_0:y_1)$ in $\PP^1(L)$. When $L$ is a finite extension of $\bQ_p$ we will take as the absolute value on $L$ the unique extension of the $p$-adic absolute $\abs{\cdot}_p$ to $L$, normalized so that $\abs{p}_p = 1/p$, and when $L=\bR$ or $L=\bC$, we will take the usual absolute value $\abs{\cdot}_\infty=\abs{\cdot}$. Let us denote by
\[
 V_\delta(E_p) = \inf_{\nu} I_\delta(\nu)
\]
this infimum over Borel probability measures supported on $E_p$, called the \emph{$\delta$-Robin constant} of $E_p$. As was demonstrated in \cite{F-P-EIOLF}, for compact sets $E_p\subset \bP^1(\bQ_p)$ there is a unique measure $\mu_{E_p}$, which minimizes the energy functional $I_\delta$ and is supported on $E_p$, that is, for which $V_\delta(E_p) = I_\delta(\mu_{E_p})$. We have the following result, which is a scholium of Theorem 3 of \cite{F-P-EIOLF}:
\begin{thm}\label{thm:global}
Let $S$ be a finite set of rational primes, and for each $p\in S$, let $L_p/\bQ_p$ be a finite normal extension and $E_p$ a closed subset of the projective line $\bP^1(L_p)$. Let $G=\Gal(\Qbar/\bQ)$ denote the absolute Galois group, and 
\[
 A_S = \{\al\in \Qbar : G\al \subset E_p\text{ for every }p\in S\}.
\]
Then
\begin{equation}\label{eqn:lower-bound}
 \liminf_{\al\in A_S} h(\al) \geq \frac{1}{2} \sum_{p\in S} V_\delta(E_p),
\end{equation}
where $V_\delta(E_p)$ is the $\delta$-Robin constant as defined above. Further, if there exists a sequence $\{\alpha_n\}_{n=1}^\infty\subset A_S$ such that the infimum above is attained, then for each $p\in S$, the probability measures on $\bP^1(\bC_p)$ distributed equally on each Galois conjugate of $\alpha_n$ must converge weakly to the unique $\delta$-equilibrium measure of $E_p$.
\end{thm}

We note that it is very easy to see, using Rumely's Fekete-Szeg\H{o} theorem with splitting conditions \cite{RumelyFeketeI,RumelyFeketeII}, that in many cases of interest $A_S$ is infinite, so that the result above is non-trivial. In particular, if $E_p=\bQ_p$ for some finite prime $p$ and every other $E_p$ contains at least a disc or interval, then $A_S$ is infinite.

We will prove our main theorem in this paper by determining the value of $V_\delta(E_p)$ for certain real intervals and $p$-adic discs of interest. Specifically, we prove the following:
\begin{thm} \label{segment}
Let $E_\infty = [-r,r]$. The equilibrium measure is absolutely continuous with respect to Lebesgue measure on $[-r,r]$, 
\[
\frac{d\mu_{E_\infty}}{dx}(x) = \begin{cases}
                                 \displaystyle \frac{2\arcsin(1/r)}{\pi^2 \sqrt{r^2-x^2}} + \frac{1}{\pi^2 x} \log\left|\frac{(x+1)(r^2-x+\sqrt{r^2-x^2}\sqrt{r^2-1})}{(x-1)(r^2+x+\sqrt{r^2-x^2} \sqrt{r^2-1})}\right|&\text{if }r\geq 1,\\
                                 \displaystyle \frac{1}{\pi\sqrt{r^2 - x^2}}&\text{otherwise,}
                                \end{cases}
\]
and the $\delta$-Robin constant is given by
\[
V_\delta(E_\infty) = \begin{cases}
                      \displaystyle \log\frac{2}{r} + \frac{2}{\pi} \int_1^r \frac{\log{x}\,dx}{\sqrt{r^2-x^2}} + 2 \int_1^r \log{x}\,d\mu_{E_\infty}(x)&\text{if }r\geq 1,\\
                      r/2 &\text{otherwise.}
                     \end{cases}
\]
\end{thm}
The proof of Theorem \ref{segment} will be given in Section \ref{sec:real} below. 
\begin{thm}
 Let $E_p = p^n \bZ_p$ for $n\in\bZ$. If $n\geq 0$, then $\mu_p$ is the normalized Haar measure of $E_p$ as an additive group, while if $n<0$, then $\mu_p$ can be written as a linear combination of the normalized Haar measures of $\bZ_p$ and $p^k \bZ_p^\times$ for $-1\geq k \geq n$, and 
 \[
  V_\delta(E_p) = \begin{cases}
                     \displaystyle n \log p + \frac{p\log p}{p-1}&\text{if }n\geq 0,\\
                     \displaystyle \frac{p+p^{2n}}{p^2 - 1} \log p&\text{if }n<0.
                    \end{cases}
 \]
\end{thm}
For $n<0$ the expression of the $\delta$-equilibrium measure $\mu_p$ as a combination of the Haar measure of $\bZ_p$ and the Haar measures of the shells $p^k\bZ_p^\times$ for $k<0$ can be computed explicitly. Those details, as well as a more general statement of the above theorem for arbitrary finite extensions of $\bQ_p$, can be found in Theorem \ref{thm:main-padic-full} of Section \ref{sec:p-adic} below.

\subsection{Example applications}
In order to give an indication of the strength of these results, we give here some example applications. We begin with an example inspired by that from \cite{F-P-EIOLF}.
\begin{example}
Let $S=\{2,\infty\}$, $E_2 = 2^{-1}\bZ_2$ and $E_\infty = [-2,2]$, so that $A_S$ is the set of all numbers which are algebraic numbers that satisfy:
\begin{itemize}
 \item All conjugates of $\alpha$ are real and lie in the interval $[-2,2]$.
 \item All conjugates of $\alpha$ have $2$-adic absolute value at most $2$.
\end{itemize}
It then follows from Theorem \ref{thm:global} that
\[
 \liminf_{\alpha\in A_S} h(\alpha) \geq \frac12 V_\delta(E_\infty) + \frac12 V_\delta(2^{-1}\bZ_2) = 0.239632\ldots + 0.25993\ldots = 0.499562\ldots 
\]
This improves on the values one could obtain from \cite[Theorem 3]{F-P-EIOLF}, which only uses the fact that the elements of $L_S$ are totally real and totally $2$-adic, and would have allowed us to conclude that
\[
 \liminf_{\al\in A_S} h(\al) \geq \frac{7\zeta(3)}{4\pi^2} + \frac{1}{2}\cdot\frac{2\log 2}{2^2-1} = 0.231049\ldots + 0.213139\ldots = 0.444188\ldots 
\]
Both results are better than the $p$-adic bound of the Bombieri-Zannier at $p=2$, which would yield
\[
 \liminf_{\al\in A_S} h(\al)\geq \frac{1}{2} \cdot \frac{\log 2}{2+1} = 0.115525\ldots
\]
and separate the bound of Schinzel for totally real numbers, which implies that
\[
  \liminf_{\al\in A_S} h(\al)\geq \frac{1}{2}\log\bigg(\frac{1+\sqrt{5}}{2}\bigg)=0.24061...
\]
\end{example}
\begin{example}
 Suppose $\al_n$ is a sequence of distinct algebraic numbers for which the conjugates of $\al_n$ all lie in $[-2,2]$ for any archimedean place. If the $\al_n$ are assumed to be algebraic integers, then as is known that the $\al_n$ must distribute at the real place according to the logarithmic equilibrium distribution of $[-2,2]$, which is given by 
 \[
  d\mu(x) = \frac{dx|_{[-2,2]}}{\pi\sqrt{4-x^2}}.
 \]
 It follows that
 \[
  \lim_{n\ra\infty} h(\al_n) = \int_{-2}^2 \frac{\log^+\abs{x}}{\pi\sqrt{4-x^2}}dx = 0.323066\ldots 
 \]
 On the other hand, if the assumption that the $\al_n$ are algebraic integers is dropped, then previously, the best result that could be applied was Schinzel's theorem for totally real algebraic numbers \cite{SchinzelTotReal}:
 \[
  \liminf_{n\ra\infty} h(\al_n) \geq \frac{1}{2}\log \frac{1+\sqrt{5}}{2} = 0.240606\ldots 
 \]
 It is worth noting that the archimedean contribution of our height bound for $E_\infty=[-2,2]$, $\frac12 V_\delta([-2,2]) = 0.239632\ldots$, is smaller than Schinzel's bound, however, as the previous example illustrates, it possesses the advantage that it can be appplied to non-integers and combined with $p$-adic splitting conditions to obtain stronger bounds.
\end{example}
On the other hand, for non-integers, our result even at only the archimedean place is highly non-trivial:
\begin{example}
 Let $S=\{\infty\}$ and $E_\infty = [-1,1]$, so that $A_S$ contains algebraic numbers all of whose conjugates lie in $[-1,1]$. Notice that $A_S$ cannot contain more than finitely many algebraic integers, as the logarithmic capacity of $[-1,1]$ is strictly less than $1$ by the classical Fekete-Szeg\H{o} theorem. It then follows from our theorem that
 \[
  \liminf_{\alpha \in A_S} h(\alpha) \geq \frac12 V_\delta([-1,1]) = 0.346574\ldots 
 \]
 which substantially exceeds both bounds in the previous example. 
\end{example}

Lastly, we note that we can generalize Theorem \ref{thm:global} to the case where we consider conjguates of $\alpha$ to over a base number field $K$, and for a set $S$ of places of $K$, we choose for each $v\in S$ a closed subset of the projective line $\bP^1(L_v)$ for a finite normal extension $L_v/K_v$. This changes the statement of Theorem \ref{thm:global} trivially in that the energies $V_\delta(E_v)$ are calculated as before, but each factor of $V_\delta(E_v)$ is now weighted by $N_v = [K_v:\bQ_v]/[K:\bQ]$ as in the proof of \cite[Theorem 9]{F-P-EIOLF}. 

\section{Archimedean results}\label{sec:real}
We will being by proving the results for the archimedean setting.
\begin{proof}[Proof of Theorem \ref{segment}]
We use the notation and terminology of Saff and Totik \cite{ST} in this proof. Thus we deal with a logarithmic energy problem with the external field $Q(x)=\log^+|x|$ on $[-r,r],\ r\ge 1.$. Note that this external field can be written as the negative of the logarithmic potential:
\[
Q(x) = \int \log|x-t|\,d\tau(t) = - U^\tau(x),
\]
where $d\tau(e^{i\theta})=d\theta/(2\pi)$ is the Haar (equilibrium) measure on $\TT.$ Consider the balayage $\hat{\tau}$ of the measure $\tau$ from the domain $\Omega=\overline{\CC}\setminus[-r,r]$ onto $[-r,r]$, see Section II.4 of \cite{ST}. It follows from Theorem 4.4 of \cite[p. 115]{ST} that $\hat{\tau}$ is a unit measure  supported on $[-r,r]$, whose potential satisfies
\[
U^{\hat{\tau}}(x) + Q(x) = U^{\hat{\tau}}(x) - U^\tau(x) = \int g_\Omega(t,\infty)\,d\tau(t), \quad x\in[-r,r],
\]
where $g_\Omega(t,\infty)$ is the Green function of $\Omega$ with pole at $\infty.$ Hence $\hat{\tau}$ is the equilibrium measure of $[-r,r]$ in the external field $Q$ by Theorem 3.3 of \cite[p. 44]{ST}. Thus the support of $\mu_{E_\infty}=\hat{\tau}$ is $[-r,r]$, and the above equation allows to find the measure explicitly by using well known integral equation methods. In particular, Theorem 3.1 of \cite[p. 221]{ST} states that if $f\in C[-1,1]$ is even and $f'(x)/\sqrt{1-x^2}\in L^p[-1,1]$ for some $p\in(1,2),$ then the integral equation
\[
- \int_{-1}^1 g(t) \log|x-t|\,dt = -f(x) + C_f,\quad x\in [-1,1],
\]
has a solution of the form
\[
g(t) = L[f'](t) + \frac{B_f}{\pi\sqrt{1-t^2}}, \quad\mbox{a.e. }t\in(-1,1), 
\]
where
\[
L[f'](t) = \frac{2}{\pi^2} \text{ PV} \int_0^1 \frac{\sqrt{1-t^2}\,s f'(s)}{\sqrt{1-s^2} (s^2-t^2)}\,ds, \quad\mbox{a.e. }t\in(-1,1), 
\]
and
\[
B_f = 1 - \frac{1}{\pi} \int_{-1}^1 \frac{s f'(s)}{\sqrt{1-s^2}}\,ds.
\]
Moreover, the constant $C_f$ is uniquely determined by 
\[
C_f = \frac{2}{\pi} \int_0^1 \frac{f(s)}{\sqrt{1-s^2}}\,ds + \log{2}.
\]
Scaling the problem from $[-r,r]$ to $[-1,1]$ by the linear change of variable $x=rt$, we apply the above stated result with $f(t):=\log^+|rt|,\ t\in[-1,1].$ It is immediate to see that 
\[
B_f = 1 - \frac{2}{\pi} \int_{1/r}^1 \frac{ds}{\sqrt{1-s^2}} = \frac{2}{\pi}\,\arcsin\frac{1}{r}
\]
and
\[
C_f = \frac{2}{\pi} \int_{1/r}^1 \frac{\log(rs)}{\sqrt{1-s^2}}\,ds + \log{2} = \frac{2}{\pi} \int_1^r \frac{\log{x}\,dx}{\sqrt{r^2-x^2}} + \log{2}.
\]
Thus it remains to evaluate $L[f'](t)$ explicitly, with
\[
L[f'](t) = \frac{2}{\pi^2} \text{ PV} \int_{1/r}^1 \frac{\sqrt{1-t^2} \,ds}{\sqrt{1-s^2} (s^2-t^2)}.
\]
One can verify by direct differentiation that for any fixed $t\in[-1,1]\setminus\{0\}$, the function
\[
F_t(s) = \frac{1}{\pi^2 t} \log\left|\frac{(s-t)(1+st+\sqrt{1-t^2}\sqrt{1-s^2})}{(s+t)(1-st+\sqrt{1-t^2}\sqrt{1-s^2})}\right|
\]
satisfies
\[
\frac{dF_t}{ds}(s) = \frac{2}{\pi^2} \frac{\sqrt{1-t^2}}{\sqrt{1-s^2} (s^2-t^2)},\quad s\in(-1,1),\ s\neq t.
\]
Since for $t\in(-1/r.1/r)$ the integral defining $L[f'](t)$ becomes a regular integral instead of principal value, we can evaluate it directly by using the antiderivative $F_t$:
\[
L[f'](t) = F_t(1) - F_t(1/r) = - F_t(1/r),\quad t\in(-1/r.1/r).
\]
The values of $L[f'](t)$ for $1/r\le|t|<1$ are found by using the identity
\[
\text{ PV} \int_0^1 \frac{ds}{\sqrt{1-s^2} (s^2-t^2)} = 0,
\]
see the last equation on page 225 of \cite{ST}. Indeed, it gives that 
\begin{align*}
L[f'](t) &= \frac{2}{\pi^2} \text{ PV} \int_{1/r}^1 \frac{\sqrt{1-t^2} \,ds}{\sqrt{1-s^2} (s^2-t^2)} = - \frac{2}{\pi^2} \text{ PV} \int_0^{1/r} \frac{\sqrt{1-t^2} \,ds}{\sqrt{1-s^2} (s^2-t^2)} \\ &= F_t(0) - F_t(1/r) = - F_t(1/r),\quad t\in(-1,-1/r]\cup[1/r,1).
\end{align*}
Hence the solution of the equilibrium integral equation on $[-1,1]$ is given by
\begin{align*}
g(t) &= \frac{2\arcsin(1/r)}{\pi^2 \sqrt{1-t^2}} - \frac{1}{\pi^2 t} \log\left|\frac{(1/r-t)(1+t/r+\sqrt{1-t^2}\sqrt{1-r^{-2}})}{(1/r+t)(1-t/r+\sqrt{1-t^2}\sqrt{1-r^{-2}})}\right|.
\end{align*}
Returning to the interval $[-r,r]$ by letting $x=rt$, we obtain the equation 
\begin{align*}
- \int_{-r}^r G(x) \log|x-y|\,dx = - \log^+|y| + \frac{2}{\pi} \int_1^r \frac{\log{x}\,dx}{\sqrt{r^2-x^2}} + \log\frac{2}{r},\quad y\in [-r,r],
\end{align*}
with
\begin{align*}
\frac{d\mu_{E_\infty}}{dx}(x) = G(x) = \frac{2\arcsin(1/r)}{\pi^2 \sqrt{r^2-x^2}} + \frac{1}{\pi^2 x} \log\left|\frac{(x+1)(r^2-x+\sqrt{r^2-x^2}\sqrt{r^2-1})}{(x-1)(r^2+x+\sqrt{r^2-x^2} \sqrt{r^2-1})}\right|.
\end{align*}
The $\delta$-Robin constant is found from (1.14) of \cite[p. 27]{ST} as
\begin{align*}
V_\delta(E_\infty) &= \log\frac{2}{r} + \frac{2}{\pi} \int_1^r \frac{\log{x}\,dx}{\sqrt{r^2-x^2}} + \int \log^+|x|\,d\mu_{E_\infty}(x) \\ &= \log\frac{2}{r} + \frac{2}{\pi} \int_1^r \frac{\log{x}\,dx}{\sqrt{r^2-x^2}} + 2 \int_1^r G(x) \log{x}\,dx.\qedhere
\end{align*}
\end{proof}

\section{$p$-adic results}\label{sec:p-adic}
Let $K$ be a non-archimedean local field with absolute value $\abs{\cdot}$. Let $O_K$ denote the ring of integers of $K$. In this section we compute the $\delta$-equilibrium measure $\mu_n$ of $\pi^n O_K$ and its associated $\delta$-Robin constant $V_\delta(\pi^n O_K)$. For the basic results of non-archimedean potential theory we refer the reader to \cite{RumelyBook,BakerRumelyBook}. We begin by setting our notation:
\begin{center}
\begin{tabular}{cl}
 $K$ & our non-archimedean local field\\
 $O_K$ & the ring of integers of $K$, $O_K=\{x \in K : \abs{x}\leq 1\}$\\
 $\pi$ & a uniformizing parameter of $K$\\
 $q$ & the order of the residue field $O_K/\pi O_K$\\
 $\lambda_k$ & the unit Haar measure of $\pi^n O_K$ as an additive group\\
 $\nu_k$ & the unit Haar measure of $\pi^n O_K^\times$ as a multiplicative group\\
 $\gamma_\infty(E)$ & the logarithmic capacity of a compact set $E\subset K$
\end{tabular}
\end{center}
We now state a few lemmas which we will need.
\begin{lemma}\label{lemma:capcity-of-pin-OK}
 The logarithmic equilibrium measure of $\pi^n O_K$ is its unit Haar measure $\mu_n$ and it has logarithmic capacity
 \begin{equation}\label{eqn:capcity-of-pin-OK}
  \log \gamma_\infty(\pi^n O_K) = n \logabs{\pi} + \frac{\logabs{\pi}}{q-1}
 \end{equation}
\end{lemma}
\begin{proof}
 The proof for $n=0$ can be found in Rumely \cite[Example 4.1.24]{RumelyBook}, and the general result follows by the scaling property for capacity. For the convenience of the reader, we will reproduce it here. Let $p_{\mu_n}(x)$ be the associated potential function:
 \[
  p_{\mu_n}(x) = \int_{\pi^n O_K} \logabs{x-y}\,d\mu_n(y).
 \]
 By translation invariance of $\mu_n$, we see that $p_{\mu_n}(x)=p_{\mu_n}(0)$ for every $x\in\pi^n O_K$. It follows that $\mu_n$ is the equilibrium measure. Since the value of $p_{\mu_n}(x)$ must agree with $\log \gamma_\infty(\pi^n O_K)$ quasi-everywhere on $\pi^n O_K$, and $U$ is constant on the set, we can evaluate it at any convenient point to determine the capacity. We compute:
 \begin{align*}
  \log \gamma_\infty(\pi^n O_K) &= p_{\mu_n}(0) = \int_{\pi^n O_K} \logabs{y}\,d\mu_n(y)\\
  &= \sum_{k=n}^\infty \int_{\pi^k O_K^\times} \logabs{y}\,d\mu_n(y)\\
  &= \sum_{k=n}^\infty \logabs{\pi^k}\,\cdot\, \mu_n(\pi^k O_K^\times)\\
  &= \sum_{k=n}^\infty k \logabs{\pi}\,\cdot\, \frac{1}{q^{k-n}} \frac{q-1}{q}\\
  &= n \logabs{\pi} + \frac{\logabs{\pi}}{q-1}.\qedhere
 \end{align*}
\end{proof}
\begin{lemma}\label{lemma:capcity-of-pin-OK-cross}
 The logarithmic equilibrium measure of $\pi^n O_K^\times$ is $\nu_n$ and it has capacity
 \begin{equation}\label{eqn:capcity-of-pin-OK-cross}
  \log \gamma_\infty(\pi^n O_K^\times) = n \logabs{\pi} + \frac{q \logabs{\pi}}{(q-1)^2}.
 \end{equation}
\end{lemma}
\begin{proof}
 Let $p_{\nu_n}(x)$ be the associated potential function:
 \[
  p_{\nu_n}(x) = \int_{\pi^n O_K^\times } \logabs{x-y}\,d\nu_n(y).
 \]
 By invariance of $\nu_n$ under multiplication by elements of absolute value $1$, we see that $p_{\nu_n}(x)=p_{\nu_n}(\pi^n)$ for every $x\in\pi^n O_K^\times$. It follows that $\nu_n$ is the equilibrium measure. Since the value of $p_{\nu_n}(x)$ must agree with $\log \gamma_\infty(\pi^n O_K^\times)$ quasi-everywhere on $\pi^n O_K^\times$, and $p_{\nu_n}$ is constant on the set, we can evaluate it at any convenient point to determine the capacity. We will compute the potential $p_{\nu_n}(\pi^n)$. We note that the Haar measure $\mu_n$ of the additive group $\pi^n O_K$ is given by
 \[
  \mu_n = \frac{1}{q} \mu_{n+1} + \frac{q-1}{q} \nu_n
 \]
 since $\mu_{n+1}$ and $\nu_n$ have disjoint supports $\pi^{n+1} O_K$ and $\pi^n O_K^\times$, respectively, whose union is $\pi^n O_K$, and for each $n$, the Haar measure $\mu_n$ is characterized by the fact that $\mu_n(\alpha+ \pi^{n+k}O_K) = 1/q^k$ for each $k\geq 1$ and $\alpha\in \pi^n O_K$, and $\nu_n$ is similarly characterized by the property that $\nu_n(\alpha+ \pi^{n+k}O_K)=\frac{1}{q^{k-1}(q-1)}$ for each $k\geq 1$ and $\alpha\in \pi^n O_K^\times$. We therefore write:
 \begin{align*}
  \log\gamma_\infty(\pi^n O_K) &= p_{\mu_n}(\pi^n)\\
  &= \frac{1}{q} \int_{\pi^{n+1}O_K} \logabs{\pi^n - y}\, d\mu_{n+1}(y)
  + \frac{q-1}{q} \int_{\pi^{n}O_K^\times} \logabs{\pi^n - y}\, d\nu_n(y)\\
  &= \frac{1}{q} \logabs{\pi^n} + \frac{q-1}{q} p_{\nu_n}(\pi^n).
 \end{align*}
Applying \eqref{eqn:capcity-of-pin-OK} from the previous lemma and solving for $p_{\nu_n}(\pi^n)$ gives the desired result.
\end{proof}

\begin{lemma}\label{lemma:delta-x-neq-y}
 Let $x,y\in K$, $\abs{x}\neq \abs{y}$. Then
 \[
  -\log \delta(x,y) = \begin{cases}
                       \log^+\min\{\abs{x},\abs{y}\} &\text{if }\max\{\abs{x},\abs{y}\}\geq 1\\
                       -\logabs{x-y} &\text{if }\max\{\abs{x},\abs{y}\}<1.\\
                      \end{cases}
 \]
\end{lemma}
\begin{proof}
 Recall from the ultrametric property that $\abs{x}\neq \abs{y}$ implies that $\abs{x-y}=\max\{\abs{x},\abs{y}\}$. Substituting this into the definition of $\delta(x,y)$ gives the desired result.
\end{proof}

We are now ready to state and prove the main $p$-adic result.
\begin{thm}\label{thm:main-padic-full}
 Let $K$ be a non-archimedean field with absolute value $\abs{\cdot}$, residue field of order $q$ and uniformizing parameter $\pi$, as above, and let $n\in \bZ$. If $n\geq 0$, then the $\delta$-equilibrium of $\pi^n O_K$ is the additive Haar measure $\lambda_n$ and it has energy
 \[
  V_\delta(\pi^n O_K) = -n \logabs{\pi} - \frac{\logabs{\pi}}{q-1}.
 \]
 If $n<0$, then the $\delta$-equilibrium measure of $\pi^n O_K$ is given by
 \[
  \mu_n = c_0 \lambda_0 + c_{-1} \nu_{-1} + \cdots + c_{n} \nu_{n}
 \]
 where the constants $c_0,c_{-1},\ldots,c_{n}\geq 0$ sum to $1$ and are explicitly given by a linear system, and the $\delta$-Robin constant is
 \[
  V_\delta(\pi^n O_k) = -\frac{q + q^{2n}}{q^2 - 1}\logabs{\pi}.
 \]
\end{thm}
It is interesting to note that as $n\ra -\infty$,
\[
 V_\delta(\pi^n O_K)\ra V_\delta(\bP^1(K)) = -\frac{q}{q^2-1}\logabs{\pi},
\]
which was computed in \cite[Theorem 1]{F-P-EIOLF}. (Note that $\logabs{\pi}<0$.)
\begin{proof}[Proof of Theorem \ref{thm:main-padic-full}]
 The case of $n\geq 0$ follows immediately from Lemma \ref{lemma:capcity-of-pin-OK} as $\delta(x,y)=\abs{x-y}$ in this case, so the $\delta$-energy reduces to the usual logarithmic energy with respect to infinity. 

 To see that the equilibrium measure $\mu_n$ must be defined as in the theorem statement for some constants $c_0,c_{-1},\ldots, c_n$, we note that the kernel $-\log \delta(x,y)$ is invariant under multiplication by elements $\al\in O_K^\times$, as $\delta(x,y)$ is $\ON{PGL}_2(O_K)$ invariant, and thus by the uniqueness of the equilibrium measure, the measure in each shell $\pi^k O_K^\times$ for $k\in\bZ$ must be a multiple of the Haar measure of $\pi^k O_K^\times$, which we have denoted $\nu_k$. Further, as $O_K$ is invariant by the translation $x\mapsto x+\alpha$, for $\alpha \in O_K$, and $\delta$ is likewise invariant by $O_K$-translations, again by the uniqueness of the equilibrium measure, the measure in $O_K$ must be a multiple of the additive Haar measure of $O_K$, which we have denoted $\lambda_0$. Thus $\mu_n = c_0 \lambda_0 + c_{-1}\nu_{-1}+\cdots + c_n \nu_n$ for some constants $c_k$, $0\geq k\geq n$, and these constants must be nonnegative as the measures are disjointly supported and $\mu_n$ is a probability measure, establishing the first claim.
 
 We will now proceed to determine the values of the constants $c_k$ for which the $\delta$-potential
\[
  U^{\mu_n}_\delta(x) = \int_{\pi^n O_K} -\log \delta(x,y)\,d\mu_n(y)
\]
is constant on $\pi^n O_K$. Note that $\pi^n O_K$ is compact in the Berkovich analytification of the projective line minus the Gauss point. Since the $\delta$-energy here coincides on $K$ with the logarithmic energy with respect to the Gauss point of the ambient Berkovich projective line, it follows (by the same argument as in \cite[Theorem 1(a)]{F-P-EIOLF}) that there is a unique $\delta$-equilibrium measure. Further, this measure is characterized by the fact that the associated $\delta$-potential function is constant quasi-everywhere on $\pi^n O_K$ (see \cite[Theorem 6.18, Corollary 7.21]{BakerRumelyBook}).
 
 By explicitly evaluating the potential $U^{\mu_n}_\delta(x)$ at $x=0,\pi^{-1},\ldots,\pi^n$ and equating these values, we will arrive at the equations determining the $c_k$ coefficients, $k=0,-1,\ldots, n$. (In fact, again using the invariance of our kernel and measure under multiplication $x\mapsto \alpha x$ for $\alpha\in O_K^\times$, it follows that $U^{\mu_n}_\delta$ will be constant on all of $\pi^n O_K$.) As $\delta(0,y)=1$ for $\abs{y}>1$, we see that:
 \begin{gather}\label{eqn:U-at-0}
 \begin{aligned}
  U^{\mu_n}_\delta(0) &= \int_{\pi^n O_K} -\log \delta(0,y)\,d\mu_n(y)\\
  &= \int_{O_K} -\log \delta(0,y)\,d\mu_n(y) = \int_{O_K} -\logabs{0-y}\,d\mu_n(y)\\
  &= c_0 \int_{O_K} -\logabs{0-y}\,d\lambda_0(y)\\
  &= -c_0 \log \gamma_\infty(O_K) = -\frac{c_0}{q-1} \logabs{\pi}.
 \end{aligned}
 \end{gather}
 To compute $U^{\mu_n}_\delta(\pi^k)$ for $k=-1,\ldots, n$, we first evaluate the integral separately on the domains $O_K$, and $\pi^\ell O_K^\times$ for $-1\geq \ell \geq n$, whose disjoint union gives $\pi^n O_K$. First note that by Lemma \ref{lemma:delta-x-neq-y}, $-\log \delta(\pi^k,y) = 0$ for $\abs{y}\leq 1$, so
 \begin{equation}\label{eqn:shell-0}
  \int_{O_K} -\log \delta(\pi^k,y)\,d\mu_n(y) = 0.
 \end{equation}
For the shells $\pi^\ell O_K^\times$ satisfying $0>\ell>k$, we apply Lemma \ref{lemma:delta-x-neq-y} to obtain  $-\log\delta(\pi^k,y) = \logabs{\pi^\ell}$ for every $y\in \pi^\ell O_K^\times$, so that
 \begin{equation}\label{eqn:shell-ell-ge-k}
  \int_{\pi^\ell O_K^\times} -\log \delta(\pi^k,y)\,d\mu_n(y) = c_\ell \int_{\pi^\ell O_K^\times} \logabs{\pi^\ell} \,d\nu_\ell(y) = \ell c_\ell \logabs{\pi}.
 \end{equation}
For $\pi^k O_K^\times$ itself, we compute:
\begin{equation*}
 -\log \delta(\pi^k,y) = 2 \logabs{\pi^k} - \logabs{\pi^k - y},
\end{equation*}
and as $\mu_n|_{\pi^k O_K^\times} = c_k \nu_k$ we obtain:
\begin{gather}\label{eqn:shell-k}
\begin{aligned}
 \int_{\pi^k O_K^\times} -\log \delta(\pi^k,y)\,d\mu_n(y) &= 2 c_k \logabs{\pi^k} - c_k \int_{\pi^k O_K^\times} -\logabs{\pi^k-y}\,d\nu_k(y)\\
 &= 2 c_k \logabs{\pi^k} - c_k \log \gamma_\infty(\pi^k O_K^\times)\\
 &= 2 c_k \logabs{\pi^k} - c_k \logabs{\pi^k} - \frac{c_k q}{(q-1)^2} \logabs{\pi}\\
 &= kc_k \logabs{\pi} - \frac{qc_k }{(q-1)^2} \logabs{\pi}
\end{aligned}
\end{gather}
where we used Lemma \ref{lemma:capcity-of-pin-OK-cross} to evaluate the integral on the right hand side. For the shells $\pi^\ell O_K^\times$ with $0>k>\ell\geq n$, we see that $-\log \delta(\pi^k,y)= \logabs{\pi^k}$, and so
\begin{equation}\label{eqn:shell-ell-le-k}
 \int_{\pi^\ell O_K^\times} -\log \delta(\pi^k,y)\,d\mu_n(y) = c_\ell \int_{\pi^\ell O_K^\times} \logabs{\pi^k} \,d\nu_\ell(y) = k c_\ell \logabs{\pi}.
\end{equation}
Combining equations \eqref{eqn:shell-0}, \eqref{eqn:shell-ell-ge-k}, \eqref{eqn:shell-k}, and \eqref{eqn:shell-ell-le-k} we find that
\begin{equation}\label{eqn:U-pi-k}
 U^{\mu_n}_\delta(\pi^k) = \bigg(-\frac{q c_k}{(q-1)^2} + \sum_{\ell=k}^{-1} \ell c_\ell + \sum_{\ell = n}^{k-1} k c_\ell\bigg) \logabs{\pi}.
\end{equation}
Setting $U^{\mu_n}_\delta(0)=U^{\mu_n}_\delta(\pi^k)$ for $k=-1,\ldots, n$ then gives us $n$ equations which, combined with the condition that the total mass of the measure be $1$, uniquely determine the coefficients $c_{0},\ldots, c_n$ (we cancel the common factors of $\logabs{\pi}$ on each side):
\begin{equation}\label{eqn:row-0}
 c_0 + c_{-1} + \cdots + c_n = 1,
\end{equation}
and 
\begin{equation}\label{eqn:row-k}
\frac{c_0}{q-1} -\frac{q c_k}{(q-1)^2} + \sum_{\ell=k}^{-1} \ell c_\ell + \sum_{\ell = n}^{k-1} k c_\ell=0 \quad \text{for}\quad k=-1,-2,\ldots, n.
\end{equation}
To solve, we first begin by subtracting $1/(q-1)$ times equation \eqref{eqn:row-0} from each equation \eqref{eqn:row-k} to obtain for each $k=-1,-2,\ldots, n$ the new equation:
\begin{equation}\label{eqn:row-k-2}
-\frac{q c_k}{(q-1)^2} + \sum_{\ell=k}^{-1} \left(\ell-\frac{1}{q-1}\right) c_\ell + \sum_{\ell = n}^{k-1} \left(k-\frac{1}{q-1}\right) c_\ell=-\frac{1}{q-1}. 
\end{equation}
Finally, for each $k$ we multiply the above equation \eqref{eqn:row-k-2} by
\[
 \frac{(q-1)^2}{q^{k-2 n+1}}\sum _{i=0}^{-2 (n-k)} (-q)^i
\]
and add all of the resulting equations to \eqref{eqn:row-0} to obtain
\[
 c_0 = 1 - \frac{1}{q} + \frac{1}{q^2} - \cdots + \frac{1}{q^{-2n}} = \frac{q+q^{2n}}{q+1}.
\]
Lastly, as $U^{\mu_n}_\delta$ is constant on all of $\pi^n O_K$ and must equal $V_\delta(\pi^n O_K)$ quasi-everywhere, we can compute the $\delta$-Robin constant by evaluating at $0$ using \eqref{eqn:U-at-0}:
\begin{equation}\label{eqn:delta-energy}
 V_\delta(\pi^n O_K) = U^{\mu_n}_\delta(0) = -\frac{q+q^{2n}}{q^2-1} \logabs{\pi}
\end{equation}
for $n<0$, which is the desired result.
\end{proof}

\bibliographystyle{abbrv}
\bibliography{bib}        

\begin{thebibliography}{1}

\bibitem{BakerRumelyBook}
M.~Baker and R.~Rumely.
\newblock {\em Potential theory and dynamics on the {B}erkovich projective
  line}, volume 159 of {\em Mathematical Surveys and Monographs}.
\newblock American Mathematical Society, Providence, RI, 2010.

\bibitem{BombieriZannierNote}
E.~Bombieri and U.~Zannier.
\newblock A note on heights in certain infinite extensions of {$\mathbb{Q}$}.
\newblock {\em Atti Accad. Naz. Lincei Cl. Sci. Fis. Mat. Natur. Rend. Lincei
  (9) Mat. Appl.}, 12:5--14, 2001.

\bibitem{F-P-EIOLF}
P.~Fili and C.~Petsche.
\newblock Energy integrals over local fields and global height bounds.
\newblock {\em Int. Math. Res. Not.}, 2013.
\newblock doi:10.1093/imrn/rnt250.

\bibitem{RumelyFeketeI}
R.~Rumely.
\newblock The {F}ekete-{S}zeg{\H o} theorem with splitting conditions. {I}.
\newblock {\em Acta Arith.}, 93(2):99--116, 2000.

\bibitem{RumelyFeketeII}
R.~Rumely.
\newblock The {F}ekete-{S}zeg{\H o} theorem with splitting conditions. {II}.
\newblock {\em Acta Arith.}, 103(4):347--410, 2002.

\bibitem{RumelyBook}
R.~S. Rumely.
\newblock {\em Capacity theory on algebraic curves}, volume 1378 of {\em
  Lecture Notes in Mathematics}.
\newblock Springer-Verlag, Berlin, 1989.

\bibitem{ST}
E.~B. Saff and V.~Totik.
\newblock {\em Logarithmic potentials with external fields}.
\newblock Springer-Verlag, Berlin, 1997.

\bibitem{SchinzelTotReal}
A.~Schinzel.
\newblock On the product of the conjugates outside the unit circle of an
  algebraic number.
\newblock {\em Acta Arith.}, 24:385--399, 1973.
\newblock Collection of articles dedicated to Carl Ludwig Siegel on the
  occasion of his seventy-fifth birthday. IV.

\end{thebibliography}

\end{document}